\documentclass[a4paper,10pt,leqno,english]{smfart}
\usepackage{aeguill}
\usepackage{enumerate}
\usepackage{amssymb,amsmath,latexsym,amsthm}
\usepackage[T1]{fontenc}
\usepackage{smfthm}      
\usepackage{geometry}   
\usepackage{url}   
\usepackage[frenchb, english]{babel}
\usepackage[utf8]{inputenc}   
\usepackage{mathrsfs}
\usepackage{xcolor}
\usepackage{comment}
\definecolor{violet}{rgb}{0.0,0.2,0.7}
\definecolor{rouge2}{rgb}{0.8,0.0,0.2}
\usepackage{hyperref}
\usepackage{soul}

\hypersetup{
    bookmarks=true,         
    unicode=false,          
    pdftoolbar=true,        
    pdfmenubar=true,        
    pdffitwindow=false,     
    pdfstartview={FitH},    
    pdftitle={},    
    pdfauthor={},     
    colorlinks=true,       
   linkcolor=violet,          
    citecolor=rouge2,        
    filecolor=black,      
    urlcolor=cyan}           
\numberwithin{equation}{section}

\newcommand{\R}{\mathbb{R}}
\newcommand{\CC}{\mathbb{C}}
\newcommand{\Q}{\mathbb{Q}}

\renewcommand{\d}{\partial}

\newcommand{\vp}{\varphi}

\renewcommand{\O}{\mathcal{O}}
\newcommand{\Ox}{\mathcal{O}_{X}}
\newcommand{\ep}{\varepsilon}

\newcommand{\la}{\langle}
\newcommand{\al}{\alpha}
\newcommand{\ra}{\rangle}

\renewcommand{\ge}{\geqslant}
\renewcommand{\le}{\leqslant}
\newcommand{\Ric}{\mathrm{Ric} \,}
\newcommand{\om}{\omega}

\newcommand{\omtlc}{\om_{t,lc}}
\newcommand{\omke}{\omega_{\rm KE}}
\newcommand{\ddc}{dd^c}
\newcommand{\xreg}{X_{\rm reg}}

\newcommand{\vpte}{\varphi_{t,\ep}}

\newcommand{\ute}{u_{t,\varepsilon}}

\newcommand{\vpe}{\varphi_{\varepsilon}}

\newcommand{\ue}{u_{\varepsilon}}
\newcommand{\ve}{v_{\varepsilon}}
\newcommand{\we}{w_{\varepsilon}}

\newcommand{\fe}{f_{\varepsilon}}
\newcommand{\Fe}{F_{\varepsilon}}

\newcommand{\Dlc}{D_{ lc}}
\newcommand{\Dklt}{D_{ klt}}

\newcommand{\ome}{\om_{\varepsilon}}

\newcommand{\omp}{\om_{\rm P}}
\newcommand{\ompe}{\om_{\rm P, \ep}}
\newcommand{\omc}{\om_{\chi}}

\newcommand{\Supp}{\mathrm {Supp}}
\newcommand{\tr}{\mathrm{tr}}

\newcommand{\psie}{\psi_{\ep}}

\newcommand{\MA}{\mathrm{MA}}

\renewcommand{\Re}{\mathrm{Re}}
\newcommand{\snc}{(X, D)_{\rm reg}}
\newcommand{\DD}{\mathbb D}
\newcommand{\uee}{u_{\ep, \eta}}
\newcommand{\Psie}{\Psi_{\ep}}

\newcommand{\p}{\partial}

\newtheorem*{thma}{Theorem A}
\newtheorem*{thmb}{Theorem B}

\title[On the boundary behavior of Kähler-Einstein metrics on lc pairs]{On the boundary behavior of Kähler-Einstein metrics on log canonical pairs }

\author{Henri Guenancia}
\address{Department of Mathematics \\
Stony Brook University \\
 Stony Brook, NY 11794-3651 USA}
\email{guenancia@math.sunysb.edu}
\urladdr{http://www.math.sunysb.edu/~guenancia}

\author{Damin Wu}
\address{Department of Mathematics\\
University of Connecticut\\
196 Auditorium Road,
Storrs, CT 06269-3009, USA}
\email{damin.wu@uconn.edu}

\begin{document}

\begin{abstract}
In this paper, we study the boundary behavior of the negatively curved Kähler-Einstein metric attached to a log canonical pair $(X,D)$ such that $K_X+D$ is ample. In the case where $X$ is smooth and $D$ has simple normal crossings support (but possibly negative coefficients), we provide a very precise estimate on the potential of the KE metric near the boundary $D$. In the more general singular case ($D$ being assumed effective though), we show that the KE metric has mixed cone and cusp singularities near $D$ on the snc locus of the pair. As a corollary, we derive the behavior in codimension one of the KE metric of a stable variety.
\end{abstract}
\maketitle

\section{Introduction}

This paper studies negatively curved Kähler-Einstein metrics on quasi-projective manifolds. This is of course a very broad topic which has witnessed a lot of developments since the foundational works of Aubin, Yau \cite{Aubin, Yau78}  in the compact case. Quickly after the resolution of Calabi's conjecture by Yau, many works have revolved around the (non-compact) complete case; let us mention Yau \cite{Yau4}, Cheng-Yau \cite{CY}, Mok-Yau \cite{MY}, R.Kobayashi \cite{KobR} and Tian-Yau \cite{Tia} in the negative scalar curvature case, and \cite{TY} in the Ricci-flat case to cite only a few of them. 

More recently, a lot of attention has been drawn to conical Kähler-Einstein metrics, which are non-complete metrics living on the complement of a (smooth) divisor in a compact manifold, having a very precise behavior near the divisor, cf. \cite{Maz, Jef, Don, Brendle, CGP, JMR, GP, Yao, DS}. \\

These classes of examples (in the negatively curved case) can be recast in a unified framework. Namely, when we are seeking negatively curved Kähler-Einstein metrics on the complement $X\setminus D$ of a smooth divisor $D$ (or merely with simple normal crossings) in a compact Kähler manifold $X$, then one has at some point to do an assumption on the positivity of the adjoint canonical bundle $K_X+D$. More precisely, the existence of a negatively curved KE metric with cuspidal singularities along $D$ such as in \cite{KobR, Tia} is equivalent to the ampleness of $K_X+D$. In the same vein, the existence of a negatively curved conical KE metric with cone angle $2\pi \beta$ along $D$ is equivalent to the ampleness of $K_X+(1-\beta)D$. \\

So what if now, we look at the problem from another angle? That is, what if instead of looking for Kähler-Einstein metrics on $X\setminus D$ having a prescribed behavior along $D$, we just start by assuming that the line bundle $K_X+aD$ is ample for some real number $a$, and see what kind of Kähler-Einstein metrics one can construct? Well, if $a\in (0,1]$, we end up with conical/cuspidal metrics because of what we explained above. In this paper, we will leave aside that case where $a>1$, and only study the situation where $a\in (-\infty, 1]$. More generally, take $D=\sum D_i$ be a simple normal crossings (snc) divisor, choose real numbers $a_i\in  (-\infty, 1]$, and assume that $K_X+\sum a_i D_i$ is ample. Is it possible to construct "reasonable" Kähler-Einstein metrics with negative scalar curvature on $X\setminus D$ that are naturally related to the data of the $a_i$'s?\\

This question has been studied from various points of view \cite{Wu,Wu2,G12,BG} and it seems that the framework of pluripotential theory could be the best fit as it yields a unified approach and treatment of the problems at stake. Indeed, it has been proved in \cite{BG} that given a pair $(X,D)$ where $D=\sum a_i D_i$ is a divisor with simple normal crossings support and coefficients $a_i\in (-\infty, 1]$ such that $K_X+D$ is ample, there exists a unique "weak" Kähler-Einstein metric $\omke$, smooth on $X\setminus D$ and satisfying $\Ric \omke = -\omke + \sum a_i [D_i]$ in the sense of currents. Moreover, the singularities of $\omke$ near $D$ are relatively mild as this current has finite energy, cf. \cite{GZ07}. What more do we know about $\omke$?

Well, first, if the pair is klt (i.e.,  $a_i<1$ for all $i$), then if follows from Ko\l odziej's estimate \cite{Kolo} that the metric has \textit{bounded} potentials. But as soon as some coefficient $a_i$ equals $1$, the potentials have to be unbounded. This can be seen using the Monge-Ampère formulation of the Kähler-Einstein problem which takes the form
\[(\om+\ddc \vp)^n = \frac{e^{\vp}\om^n}{\prod_{i} |s_i|^{2a_i}}\]
where $\om\in c_1(K_X+D)$ is a Kähler form, and $s_i$ is a defining section for $D_i$, whose associated line bundle we endow with a suitable hermitian metric (to get the condition on the Ricci curvature). Then, as $|s|^{-2}$ is not integrable, $\vp$ has to go to $-\infty$ near $\Dlc$ to garantee the integrability of the rhs. So one cannot expect bounded potentials. If now the divisors has only coefficients equal to $1$, then we know from \cite{KobR, Tia} that $\vp= - \sum_{i} \log (\log |s_i|^2)^2 + O(1)$, and that $\om_{\rm KE} $ has Poincaré singularities along $D$. We have an analogous expansion (i.e., loglog near $\Dlc$ + bounded term) if the coefficients are orbifold \cite{Tia}, or more generally of $a_i \in [0,1]$ \cite{G12,GP}. More generally, if the irreducible components of $D$ associated to coefficients $a_i<1$ do not meet any irreducible components of $D$ associated to coefficients $a_i=1$, the same result holds \cite{Wu2}. The first Theorem of this note aims to prove that the above expansion for the potential always holds regardless of the combinatoric of $D$:

\begin{thma}
\phantomsection
\label{thm:main}
Let $X$ be a compact Kähler manifold, $D=\sum a_i D_i$ be a divisor with simple normal crossings support having coefficients $a_i \in (-\infty, 1]$ and such that $K_X+D$ is ample. Let $\om \in c_1(K_X+D)$ a Kähler form, and let $\omke=\om+\ddc \vp_{\rm KE}$ be the Kähler-Einstein metric of $(X,D)$, i.e., $\Ric \omke = -\omke + [D]$. Then 
\[\vp_{\rm KE}   = - \sum_{a_i=1} \log \log^{2} \frac{1}{|s_i|^2}  +O(1) \]
\end{thma}

We will give two proofs of this result, both based on a approach involving Green's functions  but in different contexts. The two proofs share a common core: 
we start by partially regularizing the Monge-Ampère equation so as to make it of Poincaré-type as in \cite{KobR, Tia}, and then one will seek for uniform estimates on the potential, independent of the regularizing parameter. The lower bound is obtained using ideas involving approximate cone metrics, and already appearing in \cite{CGP, G12}. Then, using Yau's maximum principle for complete manifolds, we derive an upper bound of the potential involving $\sum_{a_i<0} a_i \log|s_i|^2$. Of course, the right hand side goes to $+\infty$ near the boundary divisor, so this estimate is not sufficient to prove Theorem A. This is where our two proofs take different paths. 

\noindent
The common idea is to estimate the supremum of the potential by its $L^1$ norm (which is controlled by the previous estimate) using Green's functions. The difficulty here is that there is no global positive Green's function for the Laplacian $\Delta$ on $X \setminus D_{lc}$  as follows from \cite{CY3} since a Poincaré-type metric has finite volume. On the other hand, one cannot use the local Green's function of $\Delta$ as in \cite{Wu2}, since the injectivity radius of the Poincar\'e metric shrinks to zero as the point tends to $D_{lc}$.  In the first approach, we pull-back the equation to some kind of universal cover to make the Poincaré metric into an euclidian one, so that one can use a standard local Green's function upstairs, derive an upper bound upstairs, and then push it back down to $M$. 

\noindent
In the second approach, we first construct a global Green's function associated with $\Delta_g-1$ on any complete Riemannian manifold $(M, g)$. For $M = X \setminus D_{lc}$ endowed with a Poincaré-type metric, we can control the asymptotic behavior of this function with sufficient precision so as to get an upper bound for our potential. \\

To go beyond Theorem A, it would be natural to expect higher order estimates on the potential of the Kähler-Einstein metric. As for Laplacian type estimates, this has already been done in \cite{G12,GP} whenever the coefficients of $D$ are non-negative. It would be challenging to extend these results to our more general setting, but if $D=-aH$ for some positive number $a$, $H$ being a smooth hypersurface, then of the main new issue is that we do not really have a global reference metric on $X\setminus H$ that would behave like $|z_1|^{2a}dz_1\wedge d\bar z_1 + \sum_{j>1}dz_j\wedge d\bar z_j$ near $H$, whenever it is locally given by $(z_1=0)$. \\

In the second part, we investigate the case of singular pairs $(X,D)$, i.e., $X$ is now a normal projective variety and $D$ an \textit{effective} Weil divisor on $X$ such that the pair $(X,D)$ has log canonical singularities. If one assumes that $K_X+D$ is ample, then we know from \cite{BG} that $(X,D)$ admits a unique Kähler-Einstein metric $\omke$ (see Section \ref{sec:sing} for the related definitions), which is smooth on $\xreg \setminus \Supp(D)$. We will study the behavior of $\omke$ near $D$, and more precisely at the points where $X$ is smooth and $D$ has simple normal crossings support:

\begin{thmb}
\label{thm2}
Let $(X, D)$ be a projective log canonical pair such that $K_X+D$ is ample. Then its Kähler-Einstein metric $\omke$ has mixed cone and cusp singularities along $D$ on the snc locus $\snc$ of the pair.
\end{thmb}

As a corollary of this theorem, we show that the Kähler-Einstein metric of a stable variety (in the sense of Koll\'ar-Sherpherd-Barron and Alexeev)  is cuspidal near the double crossing points, cf Corollary \ref{cor:stable}. \\

Let us conclude this introduction by saying that both Theorem A and Theorem B are the crucial analytic inputs in the proof of the polystability of the logarithmic tangent sheaf of a log canonical pair $(X,D)$ such that $K_X+D$ is ample, cf. \cite{G14}. 

\section{The smooth case}
 
 \subsection{The set-up}

The setting in this paper is the following one: $X$ is a smooth complex projective variety of dimension $n$, $D=\sum a_i D_i$ is a $\R$-divisor with simple normal crossing support with coefficients $a_i\in (-\infty,1]$ such that the adjoint bundle $K_X+D$ is ample (i.e., its Chern class contains a Kähler metric, or equivalently $K_X+D$ is $\Q$-linearly equivalent to a positive $\R$-linear combination of ample $\Q$-line bundles. We stress here that the coefficients of $D$ may be chosen to be negative.
We set $\Dlc:= \sum_{a_i=1} D_i$, and $\Dklt:=D-\Dlc$, and $M:=X \setminus \Supp(\Dlc)$. This notations are borrowed from birational geometry, in the sense that $(X, \Dklt)$ (resp. $(X,D)$ or $(X,\Dlc)$) is a \---log smooth\--- Kawamata log terminal (klt) pair (resp. log canonical (lc) pair). \\

As for endowing $X\setminus \Supp(\Dlc)$ with a natural Kähler-Einstein metric, the viewpoints and definitions vary according to the authors, and we will choose here the following definition which has the advantage to be globally formulated on $X$, and garantees the uniqueness of the metric thanks to the formalism developed in \cite{GZ07} and its companion papers: 

\begin{defi}
With the previous notations, we say that a closed positive current $\om\in c_1(K_X+D)$ on $X$ is a Kähler-Einstein metric for $(X,D)$ if it satisfies: 
\begin{enumerate}
\item The non-pluripolar product $ \om^n$ defines an absolutely continuous measure with respect to some smooth volume form $dV$ on $X$ and $\log(  \om^n/dV) \in L^{1}_{\rm loc}(X)$,
\item $\Ric \om = - \om +[D]$,
\item $\int_X  \om^n  = c_1(K_X+D)^n$.
\end{enumerate}
\end{defi}

This seemingly complicated definition comes from the fact that we know that $\om$ cannot have bounded potentials, hence we have to use the non-pluripolar Monge-Ampère operator \cite{GZ07, BEGZ} in order to define $\om^n$ (and thus $\Ric \om$ which is defined as $-\ddc \log \la \om^n \ra$ as soon as $1.$ is satisfied) and have a suitable formulation of the problem in terms of Monge-Ampère equations. \\

We know from \cite{BG} that there exists a unique such current $\om$; moreover, $\om$ defines a smooth Kähler-Einstein metric on $M$, and if $\theta_{\ep}\in c_1(\Dklt)$ is any smooth approximation of $[\Dklt]$, then $\om$ is the (weak) limit of the twisted Kähler-Einstein metrics $\ome$ satisfying $\Ric \ome = -\ome + \theta_{\ep} + [\Dlc]$. \\

\label{sec:a}
In this part, we assume that the pair $(X,D)$ is log smooth, and we prove Theorem A.
\subsection{The lower bound}
\label{sec:lb}
As we explained in the previous part, it follows from the results of \cite{BG} that it is sufficient to obtain uniform estimates for the potential $\vpe$ solution of
\[(\om+\ddc \vpe)^n = \prod_{a_i<1} (|s_i|^2+\ep^2)^{-a_i} \frac{e^{\vpe} \om^n}{\prod_{a_k=1} |s_k|^2}\]
At that point, it is convenient to work with the complete Poincaré metric $\omp:=\om -  \sum_{a_k=1} \log \log^2 |s_k|^2$ on $M$  (up to scaling the hermitian metrics on $\mathcal O(D_k)$, it defines indeed a smooth complete Kähler metric with bounded geometry on $M$); so we set $\ue:=\vpe +  \sum_{a_k=1} \log \log^2 |s_k|^2$, so that the equation becomes (on $M$)
\begin{equation}
\label{eq1}
(\omp+\ddc \ue)^n = \prod_{a_i<1} (|s_i|^2+\ep^2)^{-a_i} e^{\ue+F} \omp^n
\end{equation}
where $F$ is known to be a bounded smooth function on $M$ (which is even smooth in the quasi-coordinates, cf. \cite{KobR, Tia}.\\

The first step is to introduce the regularized cone metric \cite{Clodo, CGP,G12}. To sum up the construction therein, there exists a smooth $\om$-psh (and $\omp$-psh) potential $\psie$ which is uniformly bounded, and such that the metric $\ompe:= \omp+\ddc \psie$ on $M$ is complete, with bounded bisectional curvature and satisfies
\[\prod_{0<a_i<1} (|s_i|^2+\ep^2)^{-a_i} \omp^n = e^{G_{\ep}} \ompe^n\]
for some smooth function $G_{\ep}$ which is uniformly bounded in $\ep$. Therefore, setting $\ve:=\ue-\psie$, equation \eqref{eq1} becomes
\begin{equation}
\label{eq2}
(\ompe+\ddc \ve)^n = \prod_{a_i<0} (|s_i|^2+\ep^2)^{-a_i}e^{\ve+F_{\ep}} \ompe^n
\end{equation}
where $F_{\ep}=F+G_{\ep}+\psie$, and by the remarks above, $|\Fe| \le C$ for some \textit{uniform} $C>0$.\\

If we apply Yau's minimum principle \cite{Yau2} on the manifold $(M, \ompe)$, we get that $\inf \ve  \ge  - \sup F_{\ep} + \inf \sum_{a_i<0} a_i \log(|s_i|^2+\ep^2) $ and therefore
\begin{equation}
\label{eq3}
\inf_M \ue \ge -C
\end{equation}
for some uniform $C>0$. \\

\subsection{The upper bound I}
Let us get now to the upper bound. We cannot apply the same method here as one sees immediately, so we perform a change of function by setting $\we:=\ve - \sum_{a_i<0} a_i \log  (|s_i|^2+\ep^2)$. As $ \log  (|s_i|^2+\ep^2)$ is $C\om$-psh for some uniform $C$, it is also $C \ompe$-psh (up to changing $C$ eventually), and therefore 
\begin{eqnarray*}
 e^{\we+F_{\ep}} \ompe^n & =& \prod_{a_i<0} (|s_i|^2+\ep^2)^{-a_i}e^{\ve+F_{\ep}} \ompe^n \\ 
& = & \left(\ompe+\sum_{a_i<0} \ddc a_i  \log  (|s_i|^2+\ep^2) +  \ddc \we \right)^n \\
& \le & (C\ompe+  \ddc \we )^n
\end{eqnarray*}
hence the maximum principle yields $\sup \we \le - \inf \Fe + n \log C$, hence
\begin{equation}
\label{eq4}
\ue \le C + \sum_{a_i<0} a_i  \log  (|s_i|^2+\ep^2)
\end{equation}
for some uniform $C>0$. \\

Moreover, we know from \cite{BG} that $\vpe$ converges to $\vp_{\rm KE}$, the potential of the Kähler-Einstein metric of the pair $(X,D)$, which is a quasi-psh function. Hence, by Hartog's Theorem (cf. \cite[Theorem 3.2.12]{Hor2}), we know that there exists $C>0$ such that $\vpe \le C$. As a consequence, $\ue=\vpe+ \sum_{a_k=1} \log \log ^2 \frac{1}{|s_k|^2}$ is locally uniformly bounded above on $X\setminus \Dlc$.
Therefore, if we want to bound $\ue$ from above, we just need to do it locally around points at the intersection of $\Dlc$ and $\sum_{a_i>0}D_i$.

\subsection{The upper bound II}
Now that we have a partial upper bound \eqref{eq4} on $\ue$, one can derive a true upper bound using Green's functions based on ideas appearing in \cite[p. 141]{Wu2}. From now on, one can forget about the cone approximation and just remember the two bounds \eqref{eq3}-\eqref{eq4} satisfied by our potential $\ue$ solution of \eqref{eq1}.

We fix a point $p \in \Dlc$, and one may assume that $p$ admits a neighborhood $\Omega\simeq \DD^n$ where $\Dlc$ is given by $(z_1 \cdots z_r=0)$ and $\sum_{a_i<0}D_i$ by $(z_{r+1}\ldots z_s=0)$ in the holomorphic coordinates $z_1, \ldots, z_n$ induced on $\Omega$ under the identification $\Omega\simeq \DD^n$. Equation \eqref{eq4} can be reformulated as follows 
\begin{equation}
\label{eq:up2}
\ue(z) \le C\Big(1-\sum_{j=r+1}^s \log |z_j|\Big)
\end{equation}
If we knew that $\ue$ were quasi-psh, then we could derive a uniform upper bound from the inequality above and the arguments of \cite{Wu2}. But our function $\ue$ is only $\omp$-psh, so that one cannot apply these arguments unless we have a good knowledge of (local or global) Green's functions for the Poincaré metric. 
In the next section, we will build a global Green's function on $X\setminus \Dlc$ for $\Delta_{\omp}-1$, study its properties, and use it to get the desired upper bound.

But before that, we will give an alternative solution consisting in using the very particular geometry of the Poincaré metric. Indeed, the Poincaré metric behaves in some way like an euclidian one when we pull it back to some appropriate "cover". The right way to formalize this is to use the quasi-coordinates for the Poincaré metric (cf. \cite{KobR, Tia}): they are maps from an open subset $V\subset \mathbb C^n$ to $\DD_r:=(\DD^*)^r \times \DD^{n-r}$ having maximal rank everywhere. So they are just locally invertible, but these maps are not injective in general. \\
To construct such quasi-coordinates on $\DD_r$, we start from the universal covering map $\pi:\mathbb{D}\to\mathbb{D}^*$, given by $\pi(w)=e^{\frac{w+1}{w-1}}$. Formally, it sends $1$ to $0$. The idea is to restrict $\pi$ to some fixed ball $B(0,R)$ with $1/2<R<1$, and compose it (at the source) with a biholomorphism $\Phi_{\eta}$ of $\mathbb{D}$ sending $0$ to $\eta$, where $\eta$ is a real parameter which we will take close to $1$. If one wants to write an explicit formula, we set $\Phi_{\eta}(w)=\frac{w+\eta}{1+\eta w}$, so that the quasi-coordinate maps are given by
$\Psi_{\eta}=(\pi\circ \Phi_{\eta})^r\times \mathrm{Id}_{\mathbb{D}^{n-r}}:V=B(0,R)^r \times \mathbb{D}^{n-r}\to \DD_r$, i.e. $\Psi_{\eta}(v_1, \ldots, v_r, v_{r+1},\ldots, v_n)=(e^{\frac{1+\eta}{1-\eta} \frac{v_1+1}{v_1-1}}, \ldots, e^{\frac{1+\eta}{1-\eta} \frac{v_r+1}{v_r-1}} , v_{r+1},\ldots, v_n)$.\\
Once we have said this, it is easy to see that $\DD_r$ is covered by the images $\Psi_{\eta}(V)$ when $\eta$ goes to $1$. Now, an easy computation shows that $ \Psi_{\eta}^*\, \om_P$ is a Kähler metric on $V\subset \mathbb C^n$ which is uniformly (in $\eta$) quasi-isometric to the euclidian flat metric; moreover all the covariants derivatives of this metric are uniformly bounded with respect to $\eta$, but we will not need this property.\\

Let us go back to our situation. We started from an $\om_P$-psh function $\ue$ satisfying \eqref{eq:up2}. Pulling it back by $\Psi_{\eta}$, we get a smooth function $\uee:=\ue \circ \Psi_{\eta}$ on $V$ which is $\Psi_{\eta}^*\, \om_P$-psh hence (uniformly) quasi-psh by the observation above. Furthermore, as $\Psi_{\eta}$ acts trivially on the component $\DD^{n-r}$, we have 
\begin{equation}
\label{eq:up3}
\uee(v) \le C(1-\sum_{j=r+1}^s \log |v_j|)
\end{equation}
for all $v\in V$ thanks to \eqref{eq:up2}.\\

We are now in position to apply the arguments of \cite{Wu2}, so let us set up a precise framework. For $\rho>0$ large enough ($\rho>2n$ would be sufficient), we have $V\Subset B(0,\rho)$; let us also pick $1/2<R'<R$ and set $V'=B(0,R')^r \times \DD(0,1/2)^{n-r}$. As above, the images of $V'$ by $\Psi_{\eta}$ when $\eta$ goes to $1$ cover $(\DD^*)^r \times \DD(0,1/2)^{n-r}$.
We choose a cut-off function $\chi$ such that $\Supp(\chi) \Subset V$, and $\chi=1$ on $V'$ so that $d(\Supp(\nabla \chi),V')>0$. Finally, we denote by $G: \bar B(0,\rho) \times \bar B(0,\rho) \to [-\infty,0]$ the Green's function of $B(0,\rho)$. If $x\in V'$, we denote by $G_x$ the function $G(x,\cdotp)$. Then for any $x\in V'$, the function $\chi G_x$ satisfies $\Delta(\chi G_x)= \delta_x+G_x\Delta \chi  + \nabla G_x \cdotp \nabla \chi$ (this can be verified locally, first near $x$, and then away from $x$). Therefore, if $dV$ is the Lebesgue measure of $\mathbb C^n$, we have:
\[\int_{B(0,\rho)} \chi G_x \Delta \uee \,dV= \uee(x)+ \int_{\Supp(\nabla \chi)} \uee \left(G_x \Delta \chi +\nabla G_x \cdotp \nabla \chi\right)dV\]

Remember that $\uee$ is quasi-psh, so that $\Delta \uee \ge -C$. As a consequence, 
\[\uee(x)\le C \left( ||G_x||_{L^1}+ \int_{\Supp(\nabla \chi)} \uee\Big[ \left| G_x \Delta \chi\right| +\left|\nabla G_x \cdotp \nabla \chi\right| \Big]dV\right)\]
Of course, $\nabla \chi$ and $\Delta \chi$ are bounded by some constants depending only on $R',R$ and $n$. As for $G_x$ and $\nabla G_x$, these functions are bounded in terms of (negative powers of) $d(x,\cdotp)$, therefore they are uniformly (in $x$) bounded on $\Supp(\nabla \chi)$ by the above observation that $d(\Supp(\nabla \chi),V')>0$. Therefore, we have:

\begin{equation}
\label{eq:up4}
\uee(x) \le C \left( ||G_x||_{L^1} + ||\uee||_{L^1} \right) 
\end{equation}
Applying the Green-Riesz representation formula to the function $y\mapsto |y|^2$, we easily get that $||G_x||_{L^1} = (\rho^2-|x|^2)/2n\le \rho^2$. Moreover, thanks to equation \eqref{eq:up3}, we have a uniform control $||\uee||_{L^1} \le C$ (remember that $\ue$ hence $\uee$ are uniformly bounded from below already). Putting these two estimates together, we infer from \eqref{eq:up4}:
\[\uee(x) \le C\]
for some constant $C$ independent of $x\in V'$, $\eta$ and $\ep$. Pushing this inequality downwards to $(\DD^*)^r \times \DD(0,1/2)^{n-r}$, we obtain 
\[\ue \le C\]
on this latter open set, which ends the proof.

\section{A new global Green's function} 
 \label{se:NewGRE}
 In this section, we investigate the question of the existence of appropriate global Green's functions on the complete Kähler manifold $(X\setminus \Dlc, \omp)$. Adapting the arguments of \cite{SY, LiTam87}, we will construct on this Riemannian manifold a positive Green's function for the operator $\Delta -1$ whose behavior is well understood at infinity. This will enable us to give an alternative proof of Theorem A, cf. \S \ref{alt}.
 
 \subsection{Existence of the Green's function}
Let $(M^m, g)$ be a complete Riemannian manifold of real dimension $m$, and $\Delta_g$ be the Laplacian of $g$.
Similar to \cite[p. 81]{SY}, a function $G$ defined on $M \times M \setminus \textup{diag}(M \times M)$ is called a \emph{global positive Green's function} for $\Delta_g - 1$ on $M$ if $G$ satisfies the following properties:
\begin{enumerate}[\upshape (i)] 
  \item \label{it:ggre1} For any fixed $x \in M$,  $(\Delta_{g(y)} - 1)G(x,y) = 0$ and $G(x,y) > 0$, for all $y \in M$, $y \ne x$;
  \item \label{it:ggre2} $G(x,y) = G(y,x)$; 
  \item  \label{it:ggre3} As $y \to x$ for fixed $x$, $G(x,y) = 
     [(m - 2) \sigma_{m-1}]^{-1} \textup{dist}(x,y)^{2-m} \big(1 + o(1) \big)$.\\
\end{enumerate}

Here $\textup{dist}(x,y)$ denotes the geodesic distance between $x$ and $y$ in $M$ and $\sigma_{m-1}$ is the volume of the unit $(m-1)$-sphere in $\mathbb{R}^m$. 

\noindent
The following lemma constructs on \emph{any} complete Riemannian manifold a global positive Green's function for $\Delta_g - 1$. This is in sharp contrast to the global positive Green's function for $\Delta_g$ (cf. \cite{LiTam87}).
\begin{lemm} \label{le:ggre}
Let $(M, g)$ a complete Riemannian manifold. 
Then $(M, g)$ admits a global positive Green's function $G$ for $\Delta_g - 1$ on $M$. Furthermore, for any $x \in M$ and any compact set $B$ containing $x$, $G (x, y) \le C_B$ for all $y \in M\setminus B$, where the constant $C_B>0$ depending only on $B$.
\end{lemm}
\begin{proof}
The existence of $G$ follows from almost the same argument for Theorem A.1 in Schoen-Yau~\cite[p. 82]{SY} (see also Li-Tam~\cite{LiTam87}), using the monotone increasing sequence of positive Dirichlet Green's function $\{G_i\}$ on the exhaustion $\{\Omega_i\}$. The only difference is that here the operator $\Delta_g - 1$ allows us to compare $G_i$ with the constant function. 

More precisely, let $\{\Omega_i\}$ be an exhaustion of $M$ and $G_i$ be the positive Dirichlet Green's function on $\Omega_i$ (cf. \cite[p. 157]{Duf} or Remark~\ref{re:ggre}). Fix an arbitrary $x \in M$. We need to show that the monotonic sequence
\[
   m_i = \sup_{y \in \p B_r(x)} G_i(x,y)
\]
is bounded from above for all $r>0$. Suppose the contrary, i.e., there exists an $r>0$ such that $m_i \to +\infty$. Let
\[
   v_i (y) = \frac{1}{m_i} G_i(x,y) \quad \textup{for all $i \ge 1$}.
\]
By the maximum principle
\[
    v_i \le 1 \quad \textup{on $\Omega_i \setminus B_r(x)$}.
\]
On the other hand, for any $\varepsilon>0$, by property \eqref{it:ggre3} of $G$ we have
\[
   v_i (y) \le \varepsilon G_1(x,y) + 1 \quad \textup{on $\overline{B}_r(x) \setminus \{x\}$}
\]
for all sufficiently large $i$ such that $1/m_i < \varepsilon$. Applying the diagonal process we obtain that a subsequence of $v_i$ converges uniformly on compact subsets of $M \setminus \{x\}$ to a function $v$ in $M \setminus \{x\}$ satisfying
\[
    (\Delta_g - 1) v = 0 \quad \textup{in $M \setminus \{x\}$}
\]
and $0 \le v \le 1$ on $M \setminus B_r(x)$ and $v(y) \le \varepsilon G(x, y) +1$ on $\overline{B}_r(x) \setminus \{x\}$. Letting $\varepsilon \to 0$ yields
\[
   0 \le v \le 1 \quad \textup{in $M \setminus \{x\}$}.
\]
Since $\max_{\p B_r(x)} v_i = 1$ for all $i$, the function $v$ attains its maximum value 1 at an interior point of $M \setminus \{x\}$. 
Applying the maximum principle to $(\Delta_g - 1) v = 0$ at the interior point yields $v \le 0$ on $M \setminus\{x\}$, which contradicts $\max v = 1$.
  Hence, the sequence $\{m_i\}$ is bounded from above for all $r>0$. Then as in Schoen-Yau~\cite[p. 83]{SY} we apply the diagonal process to obtain a global positive Green's function $G$ on $M$.

For the second statement, for a given $x \in M$, assume that $x \in B \subset \subset\Omega_{i_0}$ for some $i_0\ge 1$. By the previous step
\[
   \sup_{i \ge i_0} \sup_{y \in \p B} G_i(x,y) \le C 
\]
where $C>0$ depends only on $B$. Applying the maximum principle yields
\[
   G_i(x,y) \le C  \quad \textup{on $M \setminus B$}.
\]
By the proof of the first statement~\cite[p. 83]{SY}, a subsequence of $G_i$ converges uniformly on compact subsets to $G$; hence, 
\[
   G(x,y) \le C \quad \textup{on $M \setminus B$}.
\] 
\end{proof}

\begin{rema} \label{re:ggre}
The proof of Lemma~\ref{le:ggre} 
makes use of  a classical fact that for a bounded domain $\Omega$ with smooth boundary in a Riemannian manifold, there exists a positive Dirichlet Green's function $G$ satisfies the properties \eqref{it:ggre1}--\eqref{it:ggre3}. 
This fact can be proved as follows:
By Duff~\cite[p. 104, 5.3]{Duf} one obtains a local fundamental solution $\gamma(P, Q)$ for operator $\Delta - 1$ on a sufficiently small neighborhood $\mathcal{U}$ of $\textup{diag}(\Omega \times \Omega)$. That is,  given $Q \in \Omega$, $\gamma(P, Q)$ is smooth and satisfies 
\[
    (\Delta - 1) \gamma(P, Q) = 0
\]
for any $P$ near $Q$ and $P \ne Q$. Furthermore, 
\[
   \gamma(P, Q) \sim \frac{1}{(m-2) \sigma_{m-1}} \textup{dist}(P, Q)^{2-m}, \quad \textup{as $\textup{dist}(P,Q) \to 0$}.
\]
Let
\[
   \Gamma(P, Q) = \eta(\textup{dist}(P,Q)/\epsilon) \gamma(P, Q),
\]
where $\eta = \eta(t)$ satisfies that $\eta\in C^{\infty}(\mathbb{R})$, $0 \le \eta \le 1$, $\eta \equiv 1$ for $0 \le t \le 1/2$ and $\eta \equiv 0$ for $t \ge 1$, and $\epsilon>0$ is a small constant such that the compact support of $\eta$ is contained in the neighborhood $\mathcal{U}$ of the diagonal. 
Fix an arbitrary $Q\in \Omega$. It follows that
\[
   (\Delta - 1) \Gamma(\cdot, Q) = -\delta_Q + F,
\]
where $F \equiv 2 \nabla \eta \cdot \nabla \gamma(\cdot, Q) + \gamma(\cdot, Q) \Delta \eta \in C^{\infty}_c(\Omega)$. We can solve 
\[
   (\Delta - 1) w   = - F \quad \textup{in $\Omega$}, \quad
   w = 0 \quad \textup{on $\p \Omega$}
\]
for a smooth function $w(\cdot, Q)$ on $\Omega$.
Then 
\[
  G(P, Q) = \Gamma(P, Q) + w(P, Q)
\]
is the desired Green's function. That $G(P, Q) = G(Q, P)$ is proven in Duff~\cite[p. 158]{Duf}. \qed
\end{rema}

From now on we let $M = X \setminus D_{lc}$. Then $(M, \omega)$ is a complete K\"ahler manifold of finite volume.
\begin{coro}
\label{co:ggre}
Let us endow $M := X \setminus D_{lc}$ with the metric $\omega$. Then $M$ admits a global positive Green's function $G$ which in particular belongs to $L^1(M)$. Furthermore, for any $x \in M$ and any $r>0$, $G(x,y)$ as a function of $y \in M$ satisfies
\[
   \|G(x,y)\|_{C^{k,\alpha}(M \setminus B_r(x))} \le C(k,\alpha) \sup_{y \in \p B_r} G(x,y) 
\]
where $C^{k,\alpha}$, $k \ge 0$, $0 < \alpha <1$, is the H\"older space in the sense of Cheng-Yau, $C(k,\alpha)>0$ is a constant depending only on $k$ and $\alpha$.
\end{coro}
\begin{proof}
That $G$ is in $L^1(M)$ follows immediately from Lemma~\ref{le:ggre} and the finiteness of volume. To see the estimate,
 note that $(M, \omega)$ has bounded geometry in terms of the quasi-coordinates. In constructing $G$ we apply the Schauder interior estimates to the Cheng-Yau's H\"older spaces, and then, using a diagonal process we can pass from $G_i$ to $G$.
\end{proof}

\subsection{Properties of the Green's function}
The following result is a slight variant of \cite[Lemma 2 p. 138]{Wu2}.
\begin{lemm}
\label{le:IBP}
Let $M = X \setminus D_{lc}$ with metric $\omega$.
Given $x \in M$ and a small ball $B \equiv B_r(x) \subset M$. For any $f, h \in C^{k,\alpha}(M \setminus B)$, $k \ge 2$, $0<\alpha < 1$, 
\[
     \int_{M \setminus B} \textup{div} (f \nabla h)  = - \int_{\p B} f \frac{\p h}{\p \nu}
\]
where the divergence $\textup{div}$ and gradient $\nabla$ are both with respect to $\omega$, and $\p B$ is oriented according to the outer unit normal $\nu$.
\end{lemm}
\begin{proof}
As in \cite[p. 138]{Wu2} we use the cutoff function $\chi_m(\rho)$ on $M$ such that $\chi_m \equiv 0$ for $\rho \le m$, $0 \le \chi_m \le 1$ for $m \le \rho \le m +1$, and $\chi_m \equiv 1$ for $\rho \ge m+1$.
Here
\[
   \rho = \sum_{i=1}^k \log (-\log |s_i|^2) \to +\infty \quad \textup{as $x \to D_{lc}$},
\]
and  $m \ge 1$ such that $B = B_r(x) \subset \subset \{\rho < m\}$. 
Write
\[
   \int_{M \setminus B} \textup{div} (f \nabla h) = \int_{M \setminus B} \textup{div} (\chi_m f \nabla h) + \int_{M \setminus B} \textup{div} [(1 - \chi_m) f \nabla h].
\]
Note that $1 - \chi_m(\rho)$ has compact support $\{\rho \le m+1\}$ in $M$ and $1 - \chi_m \equiv 1$ on $\overline{B}$. Applying the usual Stokes' theorem yields
\[
    \int_{M \setminus B} \textup{div} [ (1 - \chi_m) f \nabla h] = - \int_{\p B} (1 - \chi_m) f \nabla h = - \int_{\p B} (1 - \chi_m) f \nabla h.
\]
Then following \cite[p. 138]{Wu2} we obtain
\[
   \lim_{m \to +\infty} \int_{M \setminus B} \textup{div} (\chi_m f \nabla h) = 0
\]
by using Lebesgue's dominated convergence theorem.
\end{proof}

\begin{coro}
\label{co:IBP}
Let us endow $M = X \setminus D_{lc}$ with the metric $\omega$, and let $G$ be the global positive Green's function obtained in Corollary~\ref{co:ggre}. For any $x \in M$ and any $f \in C^{k,\alpha}(M)$ with $k \ge 2$, 
\[
   f(x) = - \int_M G (x,y) (\Delta_{\omega(y)} - 1) f(y) \omega^n. 
\]
\end{coro}

\begin{proof}
Fix an arbitrary $\varepsilon>0$. By Lemma~\ref{le:IBP},
\begin{align*}
   & \int_{M \setminus B_{\varepsilon}(x)} G(x, \cdot) (\Delta - 1) f \\
   & = \int_{M \setminus B_{\varepsilon}(x)} f (\Delta - 1) G - \int_{\p B_{\varepsilon}(x)} G \frac{\p f}{\p \nu} + \int_{\p B_{\varepsilon}(x)} f \frac{\p G}{\p \nu} \\
   & = - f(x),
\end{align*}
in view of properties \eqref{it:ggre1} and \eqref{it:ggre3} of $G$.
  Since $G(x,\cdot) \in L^1(M)$, by Lebesgue's dominated convergence theorem
  \[
      \int_M G(x,\cdot) (\Delta - 1) f = \lim_{\varepsilon \to 0^+} \int_{M \setminus B_{\varepsilon}(x)} G(x,\cdot) (\Delta - 1) f = - f(x).
  \]
\end{proof}

\subsection{An alternative proof of Theorem A}
\label{alt}
We borrow the notations of \S \ref{sec:a}: let $X$ be a compact Kähler manifold $X$ endowed with a snc divisor $\Dlc$, let us set $M:=X\setminus \Dlc$, and let us consider the Monge-Ampère equation on $M$:
\[(\om+\ddc u)^n = \frac{e^{u+F}\om^n}{\prod_{i} |s_i|^{2a_i}}\]
where $\Dklt=\sum a_i (s_i=0)$ is a divisor with snc support whose coefficients $a_i$ belong to $(-\infty, 1)$, and $\om$ is a metric with Poincaré-type singularities along $\Dlc$, the latter divisor being also assumed to have normal crossings with $\Supp(\Dklt)$. Finally, $F$ is a smooth function when read on the quasi-coordinate, i.e. 
$F \in C^{k,\alpha}(M)$ for all $k \ge 2$ and $0 < \alpha <1$.
 
 \noindent
We will assume part of the results of \S \ref{sec:a}, namely that $u$ is bounded below. We claim that by Corollary~\ref{co:ggre} and Corollary~\ref{co:IBP}, we can derive a true upper bound for $u$.\\

Indeed, let us perturb (as in \S \ref{sec:a}) the above Monge-Ampère equation by the following one, for $\varepsilon>0$,
  \[    (\omega + dd^c u_{\varepsilon})^n = e^{u_{\varepsilon} + F + f_{\varepsilon}} \omega^n \]
where  $ f_{\varepsilon} = - \sum_i a_i \log (|s_i|^2 + \varepsilon^2)$. We know that the latter equation has a unique solution $u_{\varepsilon} \in C^{k,\alpha}(M)$ for all $k \ge 2$ and $0 < \alpha <1$, and that $\ue$ converges to $u$ weakly on $X$, and smoothly on the compact sets of $M\setminus \Dklt$.

\noindent
Applying the inequality $e^t \ge 1 + t$ yields
\[   (\Delta_{\omega} - 1) u_{\varepsilon} \ge F + f_{\varepsilon}\]
Multiplying this inequality by $G(x,y)$ gives
\begin{eqnarray*}
   u_{\varepsilon}(x) 
   & =& - \int_M G(x,y) (\Delta_{\omega} - 1) u_{\varepsilon}(y) \omega^n(y) \\
   & \le & - \int_M G(x,y) (F + f_{\varepsilon}) (y) \omega^n(y) \\
   &\le & C
\end{eqnarray*}
in view of the fact that 
\[      \int_{M \setminus B} G(x,y) (- \log |s_i|^2)(y) \omega(y)^n \le C_B \int_{M \setminus B} (-\log |s_i|^2)(y) \omega(y)^n \le C \]
for each $i$.
Here $B$ is a compact set in $M$ containing $x$, and $C>0$ is a constant depending only on $n$, $a_i$ and $C_B$.

\section{The case of singular pairs}
\label{sec:sing}
The goal of this section is to explain and prove Theorem B about the behavior near the boundary divisor of the Kähler-Einstein metric associated with a log canonical pair $(X,D)$ such that $K_X+D$ is ample.

\subsection{Mixed cone and cusp singularities}

Let $(X,D)$ be a pair consisting in a complex manifold $X$ and a $\R$-divisor $D$ having simple normal crossing support and coefficients in $[0,1]$. A Kähler metric $\om$ on $X_0:=X\setminus \Supp(D)$ is said to have mixed cone and cusp (also called Poincaré) singularities along $D$ if $\om$ is locally quasi-isometric to the model
\[\om_{\rm mod}:=\sum_{j=1}^r \frac{i dz_j\wedge d\bar z_j}{|z_j|^{2(1-\beta_j)}} + \sum_{k=r+1}^s \frac{i dz_k\wedge d\bar z_k}{|z_k|^{2} \log^{2}|z_k|^{2}}+\sum_{l=r+s+1}^n i dz_l\wedge d\bar z_l \] 
whenever $(X,D)$ is locally isomorphic to $(X_{\rm mod}, D_{\rm mod})$, where $X_{\rm mod}=(\mathbb{D}^*)^r\times (\mathbb{D}^*)^s \times \mathbb{D}^{n-(s+r)}$, $D_{\rm mod}=(1-\beta_1) [z_1=0]+\cdots+(1-\beta_r) [z_r=0]+[z_{r+1}=0]+\cdots + [z_{r+s}=0]$; where $\beta_j \in (0,1)$ and $\mathbb{D}$ (resp. $\mathbb{D}^*$) is the disc (resp. punctured disc) of radius $1/2$ in $\mathbb C$.\\

In \cite{G12} (and later in full generality in \cite{GP}), it was proved that given a compact Kähler manifold $X$ and a divisor $D$ with simple normal crossing support and coefficients in $[0,1]$ such that $K_X+D$ is ample, there exists a unique Kähler metric $\om$ on $X_0$ with mixed cone and cusp singularities along $D$ such that  $\Ric \om = -\om$.

Of course this metric coincides with the Kähler-Einstein metric constructed in \cite{BG} in the more general case of singular log canonical pairs. Our goal in this second section is to generalize the result of \cite{G12, GP} to this singular setting, as we will explain in the next paragraph after recalling the necessary definitions.

\subsection{Log canonical pairs}

\begin{defi}
A log canonical pair $(X,D)$ consists of a complex normal variety $X$ and an \textit{effective} Weil divisor $D$ such that $K_X+D$ is $\Q$-Cartier, and such that for any log resolution $\pi:X'\to X$ of $(X, D)$, the coefficients $a_i$ defined by the formula $K_{X'}=\pi^*(K_X+D)+\sum a_iE_i$ satisfy $a_i \ge -1$ (here $E_i$ is either exceptional or the strict transform of a component of $D$).
\end{defi}

\begin{defi}
Let $(X,D)$ be a log pair. The simple normal crossing (snc) locus of the pair, denoted by $\snc$, is the locus of points $x\in X$ such that the pair $(X,D)$ is log smooth at $x$, i.e., such that there exists a Zariski open set $U\ni x$ satisfying that $U\subset \xreg$ and that the divisor $D_{|U}$ has simple normal crossing support. 
\end{defi}

The snc locus is a Zariski open set whose complement has codimension at least $2$ by normality of $X$. If now $X$ is projective, $(X, D)$ is log canonical and $K_X+D$ is ample, then the main result of \cite{BG} provides a unique \textit{Kähler-Einstein metric} $\omke$ with negative curvature, which is smooth on $\xreg\setminus \Supp(D)$. What about further regularity? So far, it is really hard to tell anything about the local behavior of this metric near the singular points of $X$; but if we look at what happens at points of the boundary divisor $D$ where it is smooth (or merely snc), then we have a better understanding of how $\om$ looks like. Indeed, if $(X,D)$ is klt (i.e., the coefficients $a_i$ above satisfy $a_i>-1$), it was proved first partially in \cite{G2} and then in full generality in \cite{GP} that $\omke$ has cone singularities along $D$ on the snc locus $\snc$. We now aim to generalize this result to the log canonical case:

\begin{theo}
\label{thm2}
Let $(X, D)$ be a projective log canonical pair such that $K_X+D$ is ample. Then its Kähler-Einstein metric $\omke$ has mixed cone and cusp singularities along $D$ on the snc locus $\snc$ of the pair.
\end{theo}

We can deduce from this statement how the Kähler-Einstein metric of a stable variety (i.e., a projective variety $X$ with semi-log canonical singularities such that $K_X$ is ample, cf. \cite{BG}) behaves near the double crossing points. Recall that a double crossing point is a point near which the variety is locally analytically isomorphic to $0\in \{xy=0\}\subset \CC^{n+1}$. 

\begin{coro}
\label{cor:stable}
Let $X$ be a stable variety. Then its Kähler-Einstein metric is locally quasi-isometric to a cusp near the double crossing points.
\end{coro}

Let us explain what it means. If $p$ is such a point and $\nu:X^{\nu}\to X$ is the normalization morphism, then $\nu^{-1}(p)$ consists of two distinct points $q',q''$ sitting on the conductor divisor $D^{\nu}$, and the pair $(X^{\nu}, D^{\nu})$ is log smooth at $q',q''$ (actually $D^{\nu}$ is even smooth near those points). The corollary expresses that the pull-back $\nu^*\omke$ of the Kähler-Einstein metric has cusp singularities along $D^{\nu}$ near $q'$ and $q''$.

This generalizes the picture existing for stable curves. Indeed, if $C$ is a stable curve, let $C'$ be its normalization and $D'$ be the reduced divisor on $C'$ whose support consists of the preimage of the nodes. Then $K_{C'}+D'$ is ample, and each connected component of $C'\setminus D' $ ($=C_{\rm reg}$) has a unique hyperbolic metric which has a cusp near each point in the support of $D'$.

\begin{proof}
There is not much more left to say. Indeed, with the above notations, the conductor $D^{\nu}$ is a reduced divisor; moreover, $(X^{\nu}, D^{\nu}) $ is log canonical, $K_{X^{\nu}}+D^{\nu}$ is ample, and the pair is log smooth at each point above a double crossing point. So if we apply Theoremt \ref{thm2} at those points, we get exactly the statement claimed in the corollary.
\end{proof}

\subsection{Proof of Theorem B}

\subsubsection{The set-up}
In order to keep more usual notations, we assume that  the initial log pair is $(Y, \Delta)$, and we consider a log resolution $\pi: (X,D) \to (Y, \Delta)$ of the pair. Here, $D=\sum a_i D_i$ is a divisor on $X$ with snc support, consisting of $\pi$-exceptional divisors (with arbitrary coefficients in $(-\infty, 1]$) and of the strict transforms of the components of $\Delta$ (with coefficients in $[0,1]$). The Kähler-Einstein $\omke$ for $(X,D)$, or equivalently the pull-back of the KE metric for $(Y,\Delta)$ by $\pi$ can be written as $\omke= \theta+ \ddc \vp$ where $\theta\in c_1(\pi^*(K_Y+\Delta))$ is a smooth semipositive and big form and $\vp$ is a $\theta$-psh function solving the Monge-Ampère equation
\[\MA(\vp) = \frac{e^{\vp}dV}{\prod_i |s_i|^{2a_i}}\]
where $s_i$ are non-zero sections of $\Ox(D_i)$, $|\cdot |_i$ are smooth hermitian metrics on $\Ox(D_i)$, and $dV$ is a smooth volume form on $X$. 
Let us also introduce as before the convenient notation $\Dlc := \sum_{a_k=1} D_k$.

\noindent
By \cite[Theorem 3.5]{BG} we know that the solution $\vp$ is the limit of the quasi-psh functions $\vpte$ solving

 \begin{equation}
\label{eq:reg2}
\la(\theta+t\om_0+ \ddc \vp_{t,\ep})^n\ra = \frac{e^{\vp_{t,\ep}}dV}{\prod_{a_j<1} (|s_i|^2+\ep ^2)^{a_i} \prod_{a_k=1} |s_k|^2} 
\end{equation}
where $\om_0$ is some fixed Kähler form. 

\noindent
We will divide the proof of Theorem B in three steps. In the first two, we will be dealing with the $L^{\infty}$ estimate on the potential (upper bound then lower bound), and in the last one, 
we will focus on the Laplacian estimate.

\subsubsection{The upper bound}

To find the upper bound, we mimic what we did in the case of log smooth pair, as the loss of positivity will not hinder the previous method. We set $\vp_P:= - \sum_{a_k=1}\log \log^2 |s_k|^2, \ue:=\vpte-\vp_P$, and $\we:=\ue - \sum_{a_i<0} a_i \log (|s_i|^2+\ep)^2$. Actually, $\ue$ and $\we$ depend on $t$, but we choose not to underline this dependence so as to keep the notations lighter. If $\omp$ denotes a metric with Poincaré singularities along $\sum_{a_k=1}D_k$, e.g. $\omp= \om_0+\ddc\vp_P$, then we have
\begin{equation}
\label{eq:ma2}
(\theta+t\om_0+\ddc\vp_P+ \ddc \ue)^n = \frac{e^{\we+\fe}\omp^n}{\prod_{0<a_i<1} (|s_i|^2+\ep^2)^{a_i}}
\end{equation}
for some uniformly bounded function $\fe$ on $X\setminus \Dlc$ (this function does not depend on $t$).

\noindent
The function $\we$ is bounded on the complete manifold $(X\setminus \Dlc,\omp)$ so one can apply Yau's maximum principle to this function. So let $(x_m)$ be a sequence such that $\we(x_m) \to \sup \we$, and $\ddc \we(x_m) \le \frac 1 m \om_P$. As $\ve$ satisfies $\ddc \we \ge \ddc \ue -C \om_0 \ge \ddc \ue -C' \om_P$, we have
\[\theta+t\om_0+\ddc\vp_P+ \ddc \ue \le C \om_P+\ddc \we\]
and therefore
\begin{eqnarray*}
\frac{e^{\we(x_m)+\fe(x_m)}\omp^n(x_m)}{\prod_{0<a_i<1} (|s_i(x_m)|^2+\ep^2)^{a_i}} &=&
(\theta+t\om_0+\ddc\vp_P+ \ddc \ue)^n(x_m)\\
&\le & (C \om_P+\ddc \we)^n(x_m)\\
&\le & (C+1/m) \omp^n(x_m)
\end{eqnarray*}
so that 
\[\we(x_m) \le -\fe(x_m)+\sum_{0<a_i<1} a_i \log (|s_i(x_m)|^2+\ep^2)+\log(C+1/m) \]
hence $\sup \we \le C$, or equivalently
\[\ue \le C + \sum_{a_i<0} a_i \log (|s_i|^2+\ep)^2 \]
But $\ue$ is $(\theta+t\om_0+\ddc \vp)$-psh, hence also $C \omp$-psh, so the arguments of the first part of this article can be applied the same way in this situation, and they yield:
\begin{equation}
\label{eq:sup}
\ue \le C
\end{equation}

\subsubsection{The lower bound} This is where we have to pay for the loss of positivity of $K_X+D$. 
We know that there exists an effective $\R$-divisor $E=\sum c_{\alpha}E_{\alpha}$, $\pi$-exceptional, such that $K_X+D-E=\pi^*(K_Y+\Delta)-E$ is ample. Therefore, one can find a Kähler metric $\om_0$ on $X$, non-zero sections $s_{\alpha}$ of $\Ox(E_{\alpha})$, and hermitian metrics $|\cdotp |_{\alpha}$ on these bundles such that the function $\chi= \sum c_{\alpha}\log |s_{\alpha}|^2$ satisfies:  
\[\theta +\ddc \chi = \om_0+[E] \]

Recall that in section \ref{sec:lb}, we introduced the potential $\psie$ of the regularized cone metric; it is a uniformly bounded $\om_0$-psh function on $X$, such that the metric $\ompe:= (1+t)\om_0 + \ddc \vp_P+\ddc \psie$ on $X\setminus \Dlc$ is complete, with bounded bisectional curvature and satisfies
\[\prod_{0<a_i<1} (|s_i|^2+\ep^2)^{-a_i} \omp^n = e^{G_{\ep}} \ompe^n\]
for some smooth function $G_{\ep}$ which is uniformly bounded in $\ep$ (and $t$, which is why we choose not to emphasize the dependence of $G_{\ep}$ on $t$). One should emphasize that this metric $\ompe$ has approximate cone singularities not only along the strict transform of $\Delta - \lceil \Delta \rceil $ but also along some exceptional divisors. Setting $\ve:=\ue-\psie-\chi$, equation \eqref{eq:reg2} becomes, on $X\setminus E$:
\begin{equation}
\label{eq5}
(\ompe+\ddc \ve)^n = e^{\ve+F_{\ep}} \ompe^n
\end{equation}
where $F_{\ep}=G_{\ep}+\psie+\chi-\sum_{a_i<0} a_i \log(|s_i|^2+\ep^2)$, and from the remarks above, $\sup \Fe \le C$ for some $C$ independent of $\ep$ and $t$. The job would be done if one could apply Yau's maximum principle to $\ve$ on the complete manifold $(X\setminus \Dlc,\ompe)$. But $\ve$ is not smooth along $E$, so we should be careful. Fortunately, $\ve=-\chi+O(1)$ tends to $+\infty$ near $E$, so one can run the proof of Yau's maximum principle without any change: let us first introduce, for every positive integer $m$, the function $h_m:=\ve - \frac 1 m \vp_P$. This function is smooth on $X\setminus(\Dlc \cup E)$ and tends to $+\infty$ near the boundary. Therefore, it attains its minimum at some point $x_m$ in $X\setminus(\Dlc \cup E)$. 
Then, $0 \le \ddc h_m (x_m)= \ddc \ve(x_m)- \frac 1 m \ddc \vp_P(x_m)$ so that $\ddc \ve(x_m) \ge - \frac C m \om_0(x_m) \ge -\frac{C'}m \ompe(x_m)$ as $\vp_P$ and $\psie$ are uniformly quasi-psh. Plugging this inequality into \eqref{eq5}, we find $\inf \ve \ge - \sup \Fe \ge -C$, hence
\begin{equation}
\label{eq:inf}
 \ue \ge C+\chi 
 \end{equation}

\subsubsection{The Laplacian estimate}
The metric $\ompe$ on $X\setminus \Dlc$ is complete and has bounded curvature, but when $\ep$ goes to zero, its curvature may blow up (in both directions) due to the conic part. 
In \cite{GP}, a new Laplacian estimate has been introduced to deal specifically with that kind of geometries (cf. Section 6.3). More precisely, if we write $\om:= \ompe$ and $\om'= \ompe+\ddc \ve$, then we get from \eqref{eq5} that $\om'^n=e^{\ve+\Fe}\om^n$, and it is shown in \cite{GP} that there exists a smooth and uniformly bounded function $\Psie$ on $X$ satisfying on $X\setminus (\Dlc \cup E)$:
\[\Delta_{\om'}(\log \tr_{\om}\om'+\Psie) \ge -C \tr_{\om'}\om\]
for some constant $C$ \textit{independent of} $\ep$. This constant takes into account a lower bound for the $\om$-Laplacian of $\ve+\Fe$ (the existence of this bound is also proved in \cite{GP}).
As $\om'=\om + \ddc \ve$, we infer:
\[\Delta_{\om'}(\log \tr_{\om}\om'+\Psie-(C+1)\ve) \ge  \tr_{\om'}\om-n(C+1)\]
The function inside the Laplacian is smooth on $X\setminus (\Dlc \cup E)$ and tends to $-\infty$ near $E$. Therefore, one can apply the same maximum principle as in the last subsection: introduce $H:=\log \tr_{\om}\om'+\Psie-(C+1)\ve$ and $H_m:=H+\frac 1 m \vp_P$. By construction $H_m$ tends to $-\infty$ near $\Dlc$ and $E$, so we can choose a point $x_m$ outside of these divisors where $H_m$ attains its maximum. At this point, we have $0 \ge \ddc H_m = \ddc H+ \frac 1 m \ddc \vp$ so that at this point again, we find $\ddc H \le \frac C m \om_0 \le \frac{C_{\ep}}{m} \om'$. 
Using a basic inequality, we find 
\begin{equation}
\label{eq:log}
\log \tr_{\om}\om'(x_m) \le (\ve+\Fe)(x_m)+(n-1)\log\left[n(C+1)+nC_{\ep}/m\right]
\end{equation}
Therefore, as $\Psie$ is uniformly bounded and $\ve$ is uniformly bounded below, we have:
\begin{eqnarray*}
\log \tr_{\om}\om' &= & H+(C+1)\ve - \Psie \\
& \le & \sup_m H(x_m) + (C+1)\ve +C \\
& \underset{\eqref{eq:log}}{\le} & \sup_m \, \big[(\ve+\Fe)(x_m)+(n-1)\log\left[n(C+1)+nC_{\ep}/m\right]\\
& & -(C+1)\ve)(x_m)\big]+(C+1)\ve +C\\
&\le & -C \inf \ve +\sup \Fe + (C+1)\ve + C\\
& \underset{\eqref{eq:sup}-\eqref{eq:inf}}{\le}& C-C\chi
\end{eqnarray*}
So in the end, we have proved that the approximate KE metric $\ompe+ \ddc \ve$ satisfies on $X\setminus E$:
\[C^{-1}e^{C\chi} \ompe \le \ompe + \ddc \ve \le Ce^{-C\chi} \ompe\]
for some constant $C>0$ independent of $\ep$ and $t$. As $\chi$ is locally bounded on $X\setminus E= \pi^{-1}((Y,\Delta)_{\rm reg})$, this ends the proof of Theorem B.

\bibliographystyle{smfalpha}
\bibliography{biblio}

\end{document}